\documentclass[14pt, reqno]{amsart}
\usepackage{graphicx,psfrag}
\usepackage[psamsfonts]{amssymb}
\usepackage{pdfsync}
\usepackage{amscd}

\usepackage{amsmath, amsthm}
\usepackage{mathrsfs}
\usepackage{stmaryrd}
\usepackage{comment}
\usepackage{enumerate}
\usepackage{extarrows}
\usepackage{xfrac}
\usepackage{cases}

\theoremstyle{plain}
    \newtheorem{theorem}{Theorem}[section]
    \newtheorem{lemma}[theorem]{Lemma}
    
    \newtheorem{proposition}[theorem]{Proposition}
\theoremstyle{definition}
    \newtheorem{definition}[theorem]{Definition}

\theoremstyle{remark}
    \newtheorem{remark}[theorem]{Remark}

\numberwithin{equation}{section}

\DeclareMathOperator{\Ind}{Ind}
\DeclareMathOperator{\Spin}{Spin}
\DeclareMathOperator{\Spinc}{\Spin^c}
\DeclareMathOperator{\Ad}{Ad}
\begin{document}

\newcommand{\EE}{\mathcal{E}}
\newcommand{\foliation}{\mathscr{F}}

\newcommand{\kt}{\mathfrak{t}}
\newcommand{\TTc}{\mathbb{T}^c}
\newcommand{\NN}{\mathcal{N}}
\newcommand{\XX}{\mathbb{X}}
\newcommand{\HH}{\mathbb{H}}
\newcommand{\LL}{\mathbb{L}}
\newcommand{\Econn}{\mathbb{E}}
\newcommand{\Fconn}{\mathbb{F}}
\newcommand{\Nconn}{\mathbb{N}}
\newcommand{\bNconn}{\overline{\mathbb{N}}}
\newcommand{\bNconnd}{\overline{\mathbb{N}}^\vee}
\newcommand{\NNconn}{\mathfrak{N}}
\newcommand{\EEconn}{\mathfrak{E}}
\newcommand{\Conn}{\mathfrak{Conn}}
\newcommand{\Edot}{E^{\bullet}}
\newcommand{\Mod}{\operatorname{Mod}}
\newcommand{\CC}{\mathcal{C}}
\newcommand{\real}{\mathbb{R}}
\newcommand{\complex}{\mathbb{C}}
\newcommand{\OO}{\mathscr{O}}
\newcommand{\PP}{\mathcal{P}}
\newcommand{\A}{\mathcal{A}}
\newcommand{\F}{\mathcal{F}}
\newcommand{\M}{\mathcal{M}}
\newcommand{\ii}{\sqrt{-1}}
\newcommand{\kh}{\mathfrak{h}}

\newcommand{\Z}{\mathbb{Z}}
\newcommand{\C}{\mathbb{C}}
\newcommand{\R}{\mathbb{R}}
\newcommand{\D}{\mathbb{D}}
\newcommand{\g}{\mathfrak{g}}
\newcommand{\uu}{\mathbf{u}}

\newcommand{\Adot}{\mathcal{A}^{\bullet}}
\newcommand{\PA}{\mathcal{P}_{A}}
\newcommand{\gext}{\operatorname{Ext}}
\newcommand{\lext}{\mathscr{Ext}}
\newcommand{\Ho}{\operatorname{Ho}}
\newcommand{\Hom}{\operatorname{Hom}}
\newcommand{\End}{\operatorname{End}}
\newcommand{\Id}{\operatorname{Id}}
\newcommand{\dbar}{\overline{\partial}}
\newcommand{\expmap}{\operatorname{exp}}
\newcommand{\normal}{N_{\scriptscriptstyle{X/Y}}}
\newcommand{\conormal}{N^\vee_{\scriptscriptstyle{X/Y}}}
\newcommand{\tnormal}{N_\otimes}
\newcommand{\btnormal}{\overline{N}_\otimes}
\newcommand{\btconormal}{\overline{N}^\vee}
\newcommand{\bcnormal}{\overline{\mathcal{N}}}
\newcommand{\bcconormal}{\overline{\mathcal{N}}^\vee}
\newcommand{\vdiff}{\partial^{\mathcal{V}}}
\newcommand{\Dcohb}{\mathcal{D}^b_{coh}}
\newcommand{\Dcoh}{\mathcal{D}_{coh}}
\newcommand{\dbarsharp}{\dbar^\sharp}
\newcommand{\codim}{\operatorname{codim}}
\newcommand{\substr}{\operatorname{Substr}}
\newcommand{\sgn}{\operatorname{sgn}}
\newcommand{\Ins}{\operatorname{Ins}}
\newcommand{\formal}{^{\scriptscriptstyle{(\infty)}}}
\newcommand{\Xformal}{{X\formal_Y}}
\newcommand{\invlim}{\varprojlim}
\newcommand{\frechet}{Fr$\acute{e}$chet}
\newcommand{\Cinf}{$C^\infty$}
\newcommand{\embed}{i_{Y/X}}
\newcommand{\Spec}{\operatorname{Spec}_\complex}
\newcommand{\aaa}{\mathfrak{a}}
\newcommand{\Xdiag}{X_{X \times X}\formal}
\newcommand{\Shat}{\hat{S}}
\newcommand{\projtan}{\tau}
\newcommand{\liftnormal}{\rho}
\newcommand{\normalconn}{\nabla^\bot}
\newcommand{\normalconnbar}{\overline{\nabla}^\bot}
\newcommand{\TmapYX}{\mathcal{T}_{Y/X}}
\newcommand{\Dnormal}{\mathfrak{D}}
\newcommand{\symmconn}{\overline{\nabla}}
\newcommand{\Sym}{\operatorname{Sym}}
\newcommand{\Sh}{\operatorname{Sh}}
\newcommand{\Aut}{\operatorname{Aut}}
\newcommand{\Spf}{\operatorname{Spf}}
\newcommand{\GL}{\mathbf{GL}}
\newcommand{\Jet}{\mathcal{J}^\infty}
\newcommand{\JetX}{\mathcal{J}^\infty_X}
\newcommand{\graded}{\operatorname{gr}}
\newcommand{\Lie}{\mathcal{L}}
\newcommand{\EuclideanConn}{\nabla_\mathfrak{E}}
\newcommand{\To}{\longrightarrow}

\newcommand{\upscript}[1]{{\scriptscriptstyle{#1}}}

\newcommand{\contract}{\llcorner}
\newcommand{\normalbar}{\overline{N}}
\newcommand{\conormalbar}[1]{\normalbar^{{\upscript{\vee} #1}}}
\newcommand{\normalext}{\overline{\mathcal{N}}}
\newcommand{\conormalext}{\normalext^{\upscript{\vee}}}
\newcommand{\Ebar}{\overline{E}}

\newcommand{\Nconnbar}[1]{\bNconn^{\upscript{\vee #1}}}
\newcommand{\Econnbar}{\overline{\Econn}}

\newcommand{\beq}{\begin{equation}}
\newcommand{\eeq}{\end{equation}}

\newcommand{\Res}{\operatorname{Res}}

\newcommand{\IndG}{\Ind_G}

\setcounter{tocdepth}{1}

% \newarrow{Equal} =====

%\newcommand{\mathbb{R}}{\mathbb{R}
%\newcommand{\MM}{\mathcal{M}}
%\newcommand{\UU}{\mathcal{U}}
%\newcommand{\dd}{\:\text{d}}
%\newcommand{\mathbb{N}}{\mathbb{N}}
%\newcommand{\PP}{\mathcal{P}}
%\newcommand{\BB}{\mathcal{B}}
%\newcommand{\ZZ}{\mathbb{Z}}
%\newcommand{\mathbb{T}}{\mathbb{T}}

\title{On the Vergne conjecture}
\date{\today}
\author{Peter Hochs
\hspace{0mm}
 and Yanli Song}

\begin{abstract}
Consider a Hamiltonian action by a compact Lie group on a possibly non-compact symplectic manifold.
We give a short proof of a geometric formula for decomposition into irreducible representations of the equivariant index of a $\Spinc$-Dirac operator in this context. This formula was conjectured by Mich{\`e}le Vergne in 2006 % \cite{MR2334206} 
and proved by Ma and Zhang in 2014. 
%\cite{Zhang14}. 
\end{abstract}

\maketitle

\tableofcontents

\section{Introduction}

% \subsection*{Basic setup}
Let us consider a Hamiltonian action by a compact, connected Lie group $G$, with Lie algebra $\g$, on a possibly non-compact symplectic manifold $(M, \omega)$ with equivariant moment map $\mu : M \to \g^{*}$.  We  assume that $(M, \omega)$ is \emph{pre-quantisable}, that is, there exists a Hermitian line bundle $E$ with a $G$-invariant Hermitian connection $\nabla^{E}$ such that 
\begin{equation}
\label{line bundle}
\frac{\sqrt{-1}}{2\pi} (\nabla^{E})^{2} = \omega.
\end{equation}

We fix a $G$-invariant almost complex structure $J$ on $M$ so that 
\beq \label{eq g omega J}
g^{TM}(X, Y) = \omega(X, JY),  \hspace{5mm} X, Y \in TM,
\eeq
defines a Riemmanian metric on $M$. The almost complex structure $J$ determines a $G$-equivariant $\Z_{2}$-graded  spinor bundle
\[
S_M^{\pm} = \wedge^{0, \mathrm{even/odd}}T^{*}M,
\]
with \emph{Clifford multiplication} denoted by 
\[
c: TM \to \mathrm{End}(S_M).
\] 

Let  $\nabla^{S_M}$ be a Hermitian Clifford connection  on $S_M$, preserving $S_M^\pm$. One has the tensor product connection
\[
\nabla := \nabla^{S_M} \otimes 1 + 1 \otimes \nabla^{E}
\]
 on $S_M \otimes E$. The associated  Spin$^c$-Dirac operator $D^E$  is defined by the following composition:
\[
\Gamma(M, S_M \otimes E)  \xlongrightarrow{\nabla} \Gamma(M, T^*M \otimes S_M \otimes E) \xlongrightarrow{c} \Gamma(M, S_M \otimes E). 
\]
Here $D^E$ anticommutes with the $\Z_2$-grading, so that one has the operators
\[
 D^E_{\pm}: \Gamma(M, S_M^\pm \otimes E) \to \Gamma(M, S_M^\mp \otimes E).
\] 

By identifying $\g^*\cong \g$ via an $\Ad$-invariant inner product, we also view $\mu$ as a map into $\g$. This map induces a vector field $X^{\mu}$ on $M$ via the infinitesimal action. Suppose the set $\{X^{\mu}=0\}$ of zeroes of $X^{\mu}$ is compact. In Section \ref{sec Braverman}, we review Braverman's equivariant index $\Ind_G(M, \mu)$ of the deformed Dirac operator
\[
D_{\mu} := D^E- \sqrt{-1}\cdot f \cdot c(X^{\mu}),
\]
where $f \in C^{\infty}(M)^G$ satisfies certain growth conditions.

Let $\lambda$ be the highest weight of an irreducible representation $\pi_{\lambda}$ of $G$, for choices of a maximal torus and positive roots. Consider the reduced space $M_{\lambda} := \mu^{-1}(G\cdot \lambda)/G$. Suppose $\mu$ is proper. Then $M_{\lambda}$ is compact, so one can define the index $\Ind(M_{\lambda}) \in \Z$ of a Dirac operator on $M_{\lambda}$, and even make sense of this if $\lambda$ is a singular value of $\mu$. Vergne conjectured in her 2006 ICM plenary lecture \cite{MR2334206} that 
\begin{equation} \label{eq vergne conj intro}
\boxed{
\Ind_G(M, \mu)_{\lambda} = \Ind(M_{\lambda}),
}
\end{equation}
where we will always use a subscript $\lambda$ to denote the multiplicity of
 $\pi_{\lambda}$. This is a generalisation of the  quantisation commutes with reduction principle \cite{Guillemin82, Meinrenken98, Meinrenken99, Paradan01, Zhang98} to noncompact manifolds.

A special case of the Vergne conjecture, related to  discrete series representations of semi-simple Lie groups, was studied by Paradan \cite{Paradan03}. A generalisation of the Vergne conjecture,  where the set $\{X^{\mu}=0\}$ is not required to be compact, was first proved by Ma and Zhang \cite{Zhang14}. Later, Paradan gave a different proof \cite{Paradan11}.  This result was extended to $\Spinc$-manifolds in \cite{Song-spinc}.

Our goal in the current paper is to give a short proof of the Vergne conjecture, Theorem \ref{main theorem}. 
% This proof is based on ideas from \cite{Song-spinc}, which involve cobordism invariance of Braverman's index.
Many ideas we will use to prove the Vergne conjecture overlap with those used in the $\Spinc$ setting in Section 5 of \cite{Song-spinc}. The reason the argument can be simplified in the symplectic case is that one has the equality \eqref{eq vergne conj intro} for $\lambda = 0$ to begin with (see
 Theorem \ref{reduction at 0}), which is not true in the $\Spinc$ case.
% (The $\Spinc$ analogue of this result only holds for tori, seeProposition 5.8 in \cite{}.)

In \cite{Zhang14, Paradan11}, it is not asumed that $\{X^{\mu}=0\}$ is compact, just that $\mu$ is proper. This generalisation is natural, because one needs to allow noncompact vanishing sets for the crucial multiplicativity property of the index in \cite{Zhang14, Paradan11}, even if $\{X^{\mu}=0\}$ is compact for the initial moment map $\mu$. We are able to avoid this issue, and work with compact vanishing sets, by only proving multiplicativity of the invariant part of the index, as in Section \ref{sec proof}. This proof is based on Braverman's cobordism invariance, and allows us to keep the proof of the Vergne conjecture short.

\subsection*{Acknowledgements}

The authors would like to thank Maxim Braverman, Nigel Higson,  Paul-\'Emile Paradan,  Eckhard Meinrenken and Mich{\`e}le Vergne for many useful discussions. Special thanks go to Xiaonan Ma and Weiping Zhang for proposing this topic and kind help. The first author was supported by the European Union, through Marie Curie fellowship PIOF-GA-2011-299300.

%The second author would like  Maxim Braverman, Nigel Higson,  Paul-\'Emile Paradan,  Eckhard Meinrenken and Mich{\`e}le Vergne for many benefited discussions. Special thanks go to Xiaonan Ma and Weiping Zhang for proposing this topic and warm helps. 

\section{Braverman's index} \label{sec Braverman}

%We equip the Lie algebra $\g$ with an $\mathrm{Ad}$-invariant inner product so that we can identify $\g$ and its dual. Then the moment map 
%\[
%\mu : M \to \g^{*} \cong \g
%\]
%induces the Kirwan vector field, defined by 
%\[
%X^{\mu}_m := \frac{d}{dt}\big|_{t=0} \exp(- t \cdot \mu(m)) \cdot m, 
%\]
%for $m\in M$.

Let $\phi: M \to \g$ be an equivariant map (with respect to the adjoint action by $G$ on $\g$). It induces the vector field $X^{\phi}$ on $M$ defined by 
\[
X^{\phi}(m) := \left. \frac{d}{dt}\right|_{t=0} \exp(- t \cdot \phi(m)) \cdot m, 
\]
for $m\in M$.
\begin{definition}
The map $\phi$ is
 \emph{taming} 
% over $M$ 
if the vanishing set $\{X^\phi = 0\} \subseteq M$ is compact. 
\end{definition}

Consider the deformed Dirac operator
\begin{equation} \label{eq deformed Dirac}
D_{\phi} = D^E- \sqrt{-1}\cdot f \cdot c(X^\phi),
\end{equation}
where $f\colon M \to [0, \infty[$ is a $G$-invariant smooth function which grows fast enough towards infinity. This is a so-called admissible function introduced by Braverman;  for the precise growth condition see  Definition 2.6 in \cite{Braverman02}. 
In \cite{Braverman02}, Braverman defined an equivariant index of the deformed Dirac operator \eqref{eq deformed Dirac}, for taming maps $\phi$, in a more general context. He proved a cobordism invariance property of this index, which is the crucial ingredient of the arguments in this paper. 

Braverman considered \emph{complete} Riemannian manifolds. 
The manifolds we will consider (such as open subsets of a given manifold) may not be complete a priori. In such cases (possibly assuming the boundary of the manifold to be regular enough) one can make the manifold complete by rescaling the Riemannian metric by a positive, $G$-invariant function on $M$, which equals  one in a neighbourhood of the zeroes of $X^{\phi}$. (See Section 4.2 in \cite{Braverman02} for more details.) The resulting index is independent of the function used to rescale the metric. When dealing with non-complete manifolds in the context of Braverman's index, we will always tacitly perform this rescaling, and choose open sets so that this is possible. After rescaling, the Riemannian metric is only given by the symplectic form and the almost complex structure as in \eqref{eq g omega J} in a neighbourhood of the zeroes of $X^{\phi}$, but that is enough for our arguments.

The results from \cite{Braverman02} that we will use are summarised in the following theorem. We will write $\widehat{R}(G) = \Hom_{\Z}(R(G), \Z)$ for the completion of the representation ring of $G$. Let $T$ be a maximal torus of $G$, and $\kt$ its Lie algebra.  % and $\kt^*$ the dual of $\kt$. 
Let $\Lambda^*_+ \subset \kt^*$ be the set of dominant weights, for a choice of positive roots. For $\lambda \in \Lambda^*_+$, let $\pi_{\lambda}$ be the irreducible representation of $G$ with highest weight $\lambda$. We will write $D_{\phi}^{\pm}$ for the restriction of $D_{\phi}$ to $\Gamma^{\infty}(S_M^{\pm}\otimes E)$.
\begin{theorem}[\cite{Braverman02}]\label{braverman index}
If the  map $\phi$ is taming, then the deformed Dirac operator $D_{\phi}$ has the following properties. 
\begin{enumerate}
\item
The kernel of  $D_{\phi}$ decomposes, as a unitary representation of $G$, into an infinite direct sum 
\[
\mathrm{ker}D_{\phi}^\pm = \hat \bigoplus_{\lambda \in \Lambda^*_+} m_{\lambda}^\pm \cdot \pi_{\lambda},
\]
where $m_{\lambda}^\pm$ is a nonnegative integer for every $\lambda$.
\item
The index
\[
\IndG(M, \phi) := \sum_{\lambda \in \Lambda^*_+} (m_{\lambda}^+ - m_{\lambda}^-) \cdot \pi_{\lambda} \quad \in \widehat{R}(G)
\]
is independent of the choices of the admissible function $f$, and the connection $\nabla^{S_M}$.  
\item
If $U$ is a $G$-invariant open subset of $M$ so that 
\[
\{X^\phi = 0\}\subseteq U \subseteq M,
\]
then 
\[
\IndG(M, \phi) = \IndG(U, \phi|_U) \quad  \in \widehat{R}(G). 
\]
\item
If $(\phi^t)_{t\in [0,1]} : M \times [0,1] \to \g$ is a smooth family of equivariant maps, which is taming over $M \times [0,1]$, and constant in $t$ on $M\times [0, \epsilon[$ and $M\times ]1-\epsilon, 1]$ for an $\epsilon > 0$, 
then
\[
\IndG(M, \phi^0) = \IndG(M, \phi^1).
\]
\end{enumerate}
\end{theorem}

\begin{remark} \label{rem other defs}
To define a $G$-equivariant index $\IndG(M, \phi)$ for a taming map $\phi$, one can also use Atiyah's index of transversally elliptic symbols as  in \cite{Paradan01, Paradan11}, or an APS-type index as in \cite{Zhang14}. They are all consistent, see Theorem 5.5 in \cite{Braverman02} and Theorem 1.5 in \cite{Zhang14}. 
\end{remark}

In what follows, we will study Braverman's index for $\phi = \mu$, where we identify $\g \cong \g^*$ via a fixed $\Ad$-invariant inner product.

% From now on, we assume in addition that the moment map $\phi$ is proper over $M$. 

\section{The main result: Vergne's conjecture} \label{sec Vergne}

Take $\lambda \in \Lambda^*_+$.  By identifying $\kt^* \cong \ii\kt^*$ via multiplication by $\ii$, we view $\lambda$ as an element of $\kt^*$. If $\lambda$ is a regular value of the moment map $\mu$, then one can construct the Marsden-Weinstein symplectic reduction $(M_\lambda, \omega_\lambda)$, with $M_\lambda = \mu^{-1}(G \cdot \lambda)/G$ being a compact symplectic orbifold provided that $\mu$ is proper. Moreover, the pre-quantum line bundle $E$ as well as the almost complex structure induce pre-quantum line bundle $E_{\lambda}$ and almost complex structure $J_{\lambda}$ on the reduced space $(M_\lambda, \omega_\lambda)$. Hence, one can define the orbifold index \cite{Kawasaki81} $\Ind(M_{\lambda}) \in \Z$ of a Spin$^c$-Dirac operator on $M_\lambda$. 

If $\lambda$ is not a regular value of $\mu$, then one can show that for generic $\epsilon \in \g$ such that $\lambda + \epsilon \in \mu(M)$, this element $\lambda + \epsilon$ is a regular value. Furthermore, the integer $\Ind(M_{\lambda+ \epsilon})$ is independent of small enough $\epsilon$. (See Theorem 2.5 in \cite{Meinrenken99} or Theorem C in \cite{Paradan01}.) One then defines
\[
\Ind(M_{\lambda}) := \Ind(M_{\lambda + \epsilon}) \in \Z,
\]
for an $\epsilon$ as above.

\begin{theorem}\label{reduction at 0}
Let  $(M, \omega)$ be a Hamiltonian $G$-space with pre-quantum line bundle $E$ and taming moment map $\mu$. If $0 \not\in \mu(M)$, then
\[
\IndG(M, \mu)_0 = 0.
\]
If  $0 \in \mu(M)$, then
\beq \label{QR at 0}
\IndG(M, \mu)_0 = \Ind(M_0) \quad \in \Z.
\eeq
\end{theorem}
\begin{proof}
Theorem C in \cite{Paradan01} is the quantisation commutes with reduction result in the compact case. However, Paradan's arguments in Section 7 of that paper also imply this statement for noncompact manifolds. See also Theorem 4.3 in \cite{TZ99}.
\end{proof}

\begin{remark}
When the manifold $M$ is compact, the moment map $\mu$ is automatically taming and proper. Then  Theorem \ref{reduction at 0} is the Guillemin--Steinberg conjecture, %\cite{Guillemin82}, 
which was first proved by Meinrenken \cite{Meinrenken98} and Meinrenken--Sjamaar \cite{Meinrenken99}. Later, Tian--Zhang \cite{Zhang98} and Paradan \cite{Paradan01} gave different proofs. 
\end{remark}

Mich{\`e}le Vergne conjectured in her 2006 ICM plenary lecture \cite{MR2334206} that the identity (\ref{QR at 0}) holds not only for the trivial representation but for all irreducible $G$-representations.
\begin{theorem}
\label{main theorem}
Let  $(M, \omega)$ be a Hamiltonian $G$-space with pre-quantum line bundle $E$ and proper, taming moment map $\mu$. 
One has
\[
\IndG(M, \mu) = \sum_{\lambda \in \Lambda^*_+\cap \mu(M)} \Ind(M_{\lambda}) \pi_{\lambda}.
\]
\end{theorem}
One can view $\IndG(M, \mu)$ as the geometric quantisation of the action by $G$ on $(M, \omega)$. Vergne's conjecture then states that quantisation commutes with reduction in this context.

In the remainder of this paper, we give a proof of Theorem \ref{main theorem}.

%In this paper, we will prove Theorem \ref{main theorem} using Theorem \ref{reduction at 0}. It should be clear that the idea overlaps that in \cite[Section 5]{Song-spinc}, in which we prove a multiplicative property of equivariant indices on non-compact Spin$^c$-manifolds.  

%%% Proper taming maps %%%

\section{Making a taming map proper}

For any $\xi \in \g^* \cong \g$, let $\xi^M$ be the vector field induced by the infinitesimal action of $G$ on $M$. Let$\mu_{\xi} \in C^{\infty}(M)$ be the pairing of $\mu$ with $\xi$.
One has the Kostant formula
\begin{equation}
\label{kostant}
2 \pi \ii \mu_\xi = \nabla^E_{\xi^M} - \mathcal{L}^E_\xi, 
\end{equation}
for $\xi \in \g$, where $\mathcal{L}^E_\xi$ is the Lie derivative of sections of $E$.
If we choose a different $G$-invariant connection $\tilde \nabla^E$ on the pre-quantum line bundle $E$, we obtain a different  map $\tilde \mu : M \to \g^*$. We will still call such a map a moment map.

\begin{lemma}\label{property of moment map}
Let  $\tilde \mu$ be an arbitrary moment map defined as in (\ref{kostant}). Let $H \subset G$ be a closed subgroup with $\kh$ its Lie algebra. If $Z$ is a connected component of $M^H \cap \tilde \mu^{-1}(\kh)$, then $\tilde \mu$ is constant over $Z$. In particular, $ \tilde \mu(Z) \in \kh^*$ is given by the weight of the action of $H$ on the line bundle $E$ over $Z$. 
\begin{proof}
For any $\xi \in \kh$ and $m \in M^H$, we have that $\xi^M(m) = 0$. Thus, by (\ref{kostant}), 
\[
\mu_{\xi}(m) = \frac{\ii}{2\pi} \cdot \mathcal{L}^{E_m}_\xi,
\] 
which is determined by the weight of the action of $H$ on $E_m$ and is locally constant. 
\end{proof}
\end{lemma}

Let $U \subset M$ be a $G$-invariant open, relatively compact subset of $M$ such that $\mu$ is taming over $U$. Let $(X^\mu)^*$ be the dual of the vector field $X^\mu$, which is a $G$-invariant 1-form on $U$. For any $G$-invariant %non-negative 
function $\chi$ on $U$, setting
\[
\nabla^E_{\chi} = \nabla^E + 2\pi \ii \chi \cdot (X^\mu)^*
\]
defines a new connection $\nabla^E_{\chi}$ on $E$. Let $\mu_{\chi} : U \to \g^*$ be the moment map determined by $\nabla^E_{\chi}$ and (\ref{kostant}). The following proposition plays a key role.

\begin{proposition} \label{prop mu proper}
Let $V$ be a $G$-invariant, relatively compact neighbourhood of $\{X^\mu = 0 \} \cap U$ such that $\overline{V} \subset U$. We can choose the function $\chi$ so that
 \begin{enumerate}
 \item $\mu_{\chi}$ is proper;
 \item $\mu_{\chi}|_V = \mu|_V$;
 \item $\|\mu_{\chi}\| \geq \|\mu|_U\|$;
 \item the vector fields $X^{\mu}|_U$ and $X^{\mu_{\chi}}$ have the same set of zeroes.
 \end{enumerate}
\begin{proof}
Let $\{\xi_1, \ldots, \xi_{\dim {\g}}\}$ be an orthonormal  basis of $\g$. We define a map $\psi : U \to \g$ by  
\[
\psi (m):=  \sum_{j=1}^{\dim {\g}} \langle X^{\mu}(m), \xi_j^M(m)\rangle \cdot  \xi_j \in \g,
\]
for $m\in M$. Then we have
\[
\langle \psi, \mu\rangle =  \|X^\mu\|^2 \quad \text{and} \quad \mu_{\chi} =  \mu + \chi \cdot \psi.
\] 
The two maps $\mu, \psi : U \to \g$ are bounded since $U$ is relatively compact. Moreover, the assumption that $X^\mu \neq 0$ outside $V$ ensures that there exists $\epsilon >0$ so that $\|\psi \| > \epsilon$ over $U \setminus V$. Thus, $\mu_{\chi}$ is proper as long as the function $\chi$ is a proper function over $U$.

If we choose the function $\chi$ so that $\chi \equiv 0$ on $V$, the second condition is satisfied. The third condition follows directly from the following inequality
\[
\|\mu_{\chi}\|\cdot \|\mu\| \geq \langle \mu_{\chi}, \mu\rangle = \|\mu\|^2 +  \chi \cdot \|X^\mu\|^2 \geq \|\mu\|^2.  
\]

It remains to compare the vanishing set of the vector fields $X^\mu$ and $X^{\mu_{\chi}}$. First, suppose $X^{\mu}(m) = 0$. Then 
$\mu_{\chi}(m) = \mu(m)$, so
%\[
$
X^{\mu_{\chi}}(m)  = X^{\mu}(m)  = 0.
$
%\]
To prove the converse implication, note that 
\begin{equation} 
\langle X^{\mu}, X^{\mu_{\chi}} \rangle= \|X^\mu\|^2 + \chi \cdot \sum_{j=1}^{\dim {\g}} \langle X^{\mu}, {\xi_j}^M \rangle^2.
\end{equation}
The second term on the right-hand side of the above equation is non-negative provided the function $\chi$ is non-negative. Then $X^{\mu_{\chi}}(m) = 0$ implies that $X^\mu(m) = 0$. This completes the proof. 
\end{proof}

\end{proposition}

%%% Localisation %%%

\section{A localisation on product manifolds}

Suppose that $N$ is a compact Hamiltonian $G$-space with pre-quantum line bundle $F$, and moment map $\mu^F : N \to \g^*$. From now on, we will denote the moment map $\mu$ by $\mu^E$, to make the disctinction with $\mu^F$ clear.
For any map $\psi: M \to \g^*$, we abuse notation by  also denoting  the map 
$
M\times N \to \g^*,
$
mapping $(m, n) \in M\times N$ to $\psi(m)$, by $\psi$. (And similarly for maps from $N$ to $\g^*$.)

Let $U \subseteq M$ be a $G$-invariant open, relatively compact subset such that $\mu^E$ is taming over $U$. Fix a subset $V \subset U$ and a function $\chi$ as in Proposition \ref{prop mu proper}.  Let $\eta \in C^{\infty}(\R)$ be a function with values in $[0,1]$, and such that
\[
\eta(t) = \left\{\begin{array}{ll} 
0 & \text{if $t \leq 1/3$;}\\
1 & \text{if $t \geq 2/3$.}
\end{array} \right.
\]
Set $W:= U \times N \times [0,1]$, and consider the map
$
\phi: W \to \g^*
$ given by
\[
\phi(m, n, t) = \mu^E_{\chi}(m) + \eta(t)\mu^F(n),
\]
for $(m, n, t) \in W$. Since $\mu^E_\chi$ is proper and $\mu^F$ is bounded, the map $\phi$ has to be proper. 
\begin{lemma} \label{lem phi taming}
The map $\phi$ is taming.
\end{lemma}
\begin{proof}
The vanishing set of $X^{\phi}$ decomposes as
\[
\{X^{\phi}=0\} = \bigcup_H G \cdot \big(W^H \cap \phi^{-1}(\kh) \big),
\]
where $H$ runs over the stabiliser groups of the action by $T$ on $W$. Since $U$ is relatively compact in $M$, only finitely many such stabilisers occur. Hence it is enough to prove that for each stabiliser $H$, the set
\beq \label{eq WH phi H}
W^H \cap \phi^{-1}(\kh) 
\eeq
is compact.

Fix a stabiliser group $H$ of the $T$-action on $W$ and  a connected component $Z$ of $W^H \cap \phi^{-1}(\kh)$. Suppose that $\alpha_Z, \beta_Z \in \kh^*$ are the weights of the action of $H$ on the line bundles $E$ and $F$ restricted to $Z$. Then we have that
\[
\phi(Z) \subset \alpha_Z + [0,1]\beta_Z \subset \kh^*,
\]
which is a compact segment. Since the closure $\overline{W^H}$ of $W^H$ in $M\times N\times [0,1]$ is compact, it has finitely many connected components. The weight of the action by $H$ on a line bundle over $\overline{W^H}$ is constant on these connected components. So there are only finitely many elements $\alpha_Z, \beta_Z \in \kh^*$ as above (where $H$ is fixed but $Z$ may vary). Hence the set
\beq  \label{eq phi WH}
\phi\bigl(W^H \cap \phi^{-1}(\kh) \bigr)
\eeq
is compact.

The point of using the proper moment map 
 $\mu^E_{\chi}$ rather than the original map $\mu^E$ is that this makes  the map $\phi$  proper. Therefore, compactness of the set \eqref{eq phi WH} implies compactness of the set \eqref{eq WH phi H}.
\end{proof}

A particular consequence of Lemma \ref{lem phi taming} is that the map 
\[
\mu^E_{\chi} + \mu^F = \phi(\relbar, \relbar, 1)
\]
is taming. So the index
\[
\IndG(U_M \times N, \mu^E_{\chi} + \mu^F)
\]
is well-defined. By Lemma \ref{lem phi taming} and the fourth point of Theorem \ref{braverman index}, it equals 
\[
\IndG(U_M \times N, \mu^E_{\chi}).
\]
By Proposition \ref{prop mu proper}, the vector fields induced by $\mu^E_{\chi}$ and $\mu^E$ have the same set of zeroes, and are equal in a neighbourhood of that set. So by the third point of Theorem \ref{braverman index}, we find that
\beq \label{eq key-1}
\IndG(U_M \times N, \mu^E_{\chi} + \mu^F) = \IndG(U_M \times N, \mu^E) \quad \in \widehat{R}(G).
\eeq

%%% Proof %%%

\section{Proof of the Vergne conjecture} \label{sec proof}

Let us fix a $\lambda \in \Lambda^*_+$, and let $N := G\cdot \lambda$ be the orbit through $\lambda$ of the coadjoint action by $G$ on $\g^*$. Let $F$ be the dual of the canonical pre-quantum holomorphic line bundle on $N$, so that the associated moment map $\mu^F$ is minus the inclusion $N \hookrightarrow \g^*$. By the Borel--Weil--Bott theorem, we know that
\[
\IndG(N, F) = \pi_\lambda^*.
\]
%where $W_\lambda$ the irreducible $G$-representation with highest weight $\lambda$ and $W_\lambda^*$ its dual. 

Let $R > \|\lambda\|^2$ be a regular value of the function $\|\mu^E + \mu^F\|^2 : M \times N \to \R$. Define 
\[
U_{M\times N} = \{(m, n) \in M \times N\big| \|\mu^E(m) + \mu^F(n)\|^2 < R\} \subseteq M \times N.
\]
Then $U_{M\times N}$ is a $G$-invariant, open, relatively compact subset of $M \times N$. For a generic choice of $R$, the map $\mu^E + \mu^F$ is taming over $U_{M \times N}$, as we will assume. By the choice of $N$, 
\[
(\mu^E + \mu^F)^{-1}(0) \cong (\mu^E)^{-1}( G \cdot \lambda). 
\]
By Theorem \ref{reduction at 0}, we therefore have
\begin{equation}
\label{eq-1}
\IndG(U_{M\times N}, \mu^E + \mu^F)_0 = \Ind(M_\lambda) \quad \in \Z,
\end{equation}
if $\lambda \in \mu^E(M)$, and zero otherwise.

%Let $U_M \subset M$ be a $G$-invariant, open, relatively compact subset such that
%\[
%\{X^{\mu^E} = 0\} \subset U_M \subset M. 
%\]

Choose $R' > 0$ large enough so that the set
\[
U_M := \{m\in M; \|\mu^E(m)\|^2 < R'\}
\]
contains $\{X^{\mu^E} = 0\}$. Again, we can choose $R'$ such that $\mu^E$ is taming on $U_M$.
In addition, choose $R' > R$ so that there is a $G$-invariant neighbourhood $V_M$ of $\{X^{\mu^E}=0\}$ such that $\overline{V_M} \subset U_M$, and
\[
\overline{U_{M\times N}} \subset V_M \times N.
\]
This is possible because $\mu^F$ is bounded on $N$. Let the function $\chi \in C^{\infty}(U_M)^G$ be as in Proposition \ref{prop mu proper}, applied with $U = U_M$ and $V = V_M$.
By \eqref{eq key-1}, we have that
\[
\IndG(U_M \times N, \mu^E_{\chi} + \mu^F) = \IndG(U_M \times N, \mu^E) \quad  \in \widehat{R}(G). 
\]
In particular, 
\begin{equation}
\label{eq-2}
\IndG(U_M \times N, \mu^E_{\chi} + \mu^F)_0 =  \bigl(\IndG(U_M, \mu^E) \otimes \pi_\lambda^*\bigr)_0  =  \IndG(M, \mu^E)_\lambda \quad \in \Z.
\end{equation}

Because of \eqref{eq-1} and \eqref{eq-2}, the last step in the proof of Theorem \ref{main theorem} is the following equality.
\begin{lemma} \label{lem U M N}
For $R$ and $R'$ large enough, one has
\[
\IndG(U_M \times N, \mu^E_{\chi} + \mu^F)_0 = \IndG(U_{M\times N}, \mu^E + \mu^F)_0 \quad \in \Z. 
\]
\end{lemma}
\begin{proof}
Corollary 6.18 in \cite{Paradan01} (see also Theorem 9.6 in \cite{Paradan15}) implies that, for $R$ and $R'$ large enough,
\[
\IndG\bigl(U_M \times N, \mu^E_{\chi} + \mu^F \bigr)_0 = \IndG\bigl(U_{M\times N}, \mu^E_{\chi} + \mu^F \bigr)_0.
\]
This follows from the fact that $U_{M\times N}$ is a neighbourhood of the set of zeroes of $\mu^E_{\chi} + \mu^F$. Here we have used the equivalence of Braverman's index and the index defined by Paradan and Vergne (see Remark \ref{rem other defs}).
Inside $U_{M\times N}$, 
the function $\chi$ equals zero. Hence
\[
\IndG\bigl(U_{M\times N}, \mu^E_{\chi} + \mu^F \bigr) = \IndG\bigl(U_{M\times N}, \mu^E + \mu^F \bigr).
\]
\end{proof}

%Since $\chi \equiv 0$ on $V_M$, we have $\mu^E_{\chi} + \mu^F \equiv \mu^E + \mu^F$ on $V_M \times N$. Hence by Lemma \ref{local equi}, we find that
%\[
%\IndG(U_M \times N, \mu^E_{\chi} + \mu^F)_0 = \IndG(U_{M\times N}, \mu^E + \mu^F)_0.
%\]
%Combining this equality with (\ref{eq-1}) and  (\ref{eq-2}), we conclude that indeed
%\[
% \IndG(M, \mu^E)_{\lambda} = \Ind(M_\lambda) \quad \in \Z.
%\] 

%\begin{lemma}
%One can choose $U_M$ so that
%\begin{equation}
%\label{eq-3}
%\IndG(U_M \times N, \mu^E_{\chi} + \mu^F)_0 = \IndG(U_{M\times N}, \mu^E + \mu^F)_0 \in \Z. 
%\end{equation}
%\begin{proof}
%Let $C$ be a regular value of  $\|\mu^E\|^2 : M  \to \R$ and 
%\[
%U_{M} = \{m \in M \big| \|\mu^E(m)\|^2 < C \} \subseteq M.
%\]
%Here we choose $C \gg R$ so that
%\[
%\overline{U_{M \times N}} \subset U_M \times N. 
%\]
%Since the left-hand side of equation (\ref{eq-3})  does not depend on the choice of $\chi$, we can assume that $\chi \equiv 0$ on $\overline{U_{M \times N}}$. It follows that $\mu^E_{\chi} + \mu^F \equiv \mu^E + \mu^F$ on $U_{M \times N}$. One can finish the proof by Lemma \ref{local equi}. 
%\end{proof}
%\end{lemma}
%
%By (\ref{eq-1}),  (\ref{eq-2}) and  (\ref{eq-3}), we conclude that
%\[
% \IndG(M, \mu^E)_{\lambda} = \Ind(M_\lambda) \in \Z.
%\]
%This concludes the proof of Theorem \ref{main theorem}.

\begin{remark}
In the proof of Lemma \ref{lem U M N}, we used Corollary 6.18 from \cite{Paradan01}. That result states that the invariant part of the index vanishes if the norm of the moment map has a large enough lower bound. This was generalised to multiplicities of arbitrary irreducible representations in Theorem 2.1 in \cite{Zhang14} and Theorem 2.9 in \cite{Paradan11} (in the symplectic setting) and Theorem 3.4 in \cite{Song-spinc} (in the $\Spinc$ setting). Using one of the results in the symplectic setting, one can generalise the definition of the index in Theorem \ref{braverman index} to proper, non-taming moment maps (see Definition 1.3 in \cite{Zhang14} and Definition 2.10 in \cite{Paradan11}). In addition, by using a suitable version of these vanishing results, one can generalise Lemma \ref{lem U M N} to multiplicities of arbitrary irreducible representations. This can then be  used to generalise Theorem \ref{main theorem} to proper, non-taming moment maps. Because our goal was to give a short proof of Vergne's conjecture, we have not included the details of this generalisation in this paper.
%The authors do not see how to deduce the proper, but not necessarily taming, moment map case proved in \cite{Zhang14} from the Vergne conjecture in a straightforward way. If $\mu$ is proper, but possibly non-taming, then the set of zeroes of $X^{\mu}$ has compact connected components. On a neighbourhood of each component, the moment map is taming. But because it is not proper in general, the Vergne conjecture does not directly apply to such neighbourhoods.
\end{remark}

\bibliographystyle{alpha}
\bibliography{mybib}

\end{document}